\documentclass[12pt,leqno,a4paper]{amsart}
\usepackage{amssymb,enumerate}
\usepackage{hyperref}
\newcommand{\FF}{{\mathbb{F}}}
\newcommand{\QQ}{{\mathbb{Q}}}
\newcommand{\ZZ}{{\mathbb{Z}}}

\newcommand{\bC}{{\mathbf{C}}}
\newcommand{\bG}{{\mathbf{G}}}
\newcommand{\bL}{{\mathbf{L}}}
\newcommand{\bM}{{\mathbf{M}}}
\newcommand{\bS}{{\mathbf{S}}}
\newcommand{\bT}{{\mathbf{T}}}
\newcommand{\bZ}{{\mathbf{Z}}}

\newcommand{\fS}{{\mathfrak{S}}}

\newcommand{\cE}{{\mathcal{E}}}

\newcommand{\cent}{\operatorname{C}}
\newcommand{\Ind}{{\operatorname{Ind}}}
\newcommand{\Irr}{{\operatorname{Irr}}}
\newcommand{\norm}{\operatorname{N}}
\newcommand{\Sub}{\operatorname{Sub}}
\newcommand{\Uch}{{\operatorname{Uch}}}

\newcommand{\GL}{{\operatorname{GL}}}
\newcommand{\SL}{{\operatorname{SL}}}
\newcommand{\SU}{{\operatorname{SU}}}
\newcommand{\GU}{{\operatorname{GU}}}
\newcommand{\Sp}{{\operatorname{Sp}}}
\newcommand{\GO}{{\operatorname{GO}}}
\newcommand{\SO}{{\operatorname{SO}}}

\newcommand{\tw}[1]{{}^{#1}\!}
\newcommand{\hbS}{{\hat\bS}}

\newcommand\subn{{\,\triangleleft\triangleleft\,}}
\newcommand{\blangle}{{\big\langle}}
\newcommand{\brangle}{{\big\rangle}}

\let\eps=\epsilon
\let\ga=\gamma
\let\la=\lambda
\let\vhi=\varphi

\theoremstyle{theorem}
\newtheorem{thm}{Theorem}[section]
\newtheorem{lem}[thm]{Lemma}
\newtheorem{prop}[thm]{Proposition}

\newtheorem{thmA}{Theorem}
\newtheorem{conjA}[thmA]{Conjecture}

\theoremstyle{definition}
\newtheorem{rem}[thm]{Remark}

\makeatletter
\def\namedlabel#1#2{\begingroup #2%
    \def\@currentlabel{#2}%
    \phantomsection\label{#1}\endgroup
}

\begin{document}

\title[The subnormaliser conjecture]{The subnormaliser conjecture\\ and unipotent characters}

\author{Gunter Malle}
\address{FB Mathematik, RPTU Kaiserslautern--Landau, Postfach 3049,
  67653 Kaisers\-lautern, Germany.}
\email{malle@mathematik.uni-kl.de}

\begin{abstract}
We prove instances of the subnormaliser conjecture on character
bijections for finite groups by deriving generic versions of such bijections
for unipotent characters of nearly simple groups of Lie type. For this, we
formulate an extension of $d$-Harish-Chandra theory to what we call
\emph{generic subnormalisers}, which are certain, usually disconnected,
reductive subgroups of a simple algebraic group. For very good primes the
generic bijections give rise to bijections to certain subgroups that should
contain subnormalisers and thus be suitable for an eventual inductive approach.
If the group in question has abelian Sylow $\ell$-subgroups for some
prime~$\ell$ then our bijections satisfy the properties predicted by the
subnormaliser conjecture, and moreover preserve character values up to sign,
Brauer $\ell$-blocks, and are Galois equivariant. 
\end{abstract}


\subjclass[2020]{}

\thanks{The author gratefully acknowledges financial support by the
 DFG -- Project 286237555 -- TRR 195.}

\date{\today}

\maketitle


\section{Introduction}   \label{sec:intro}

Character correspondences are a recurrent and central theme in representation
theory of finite groups. A fundamentally new type was recently proposed in
the manuscript \cite{MR25} of Alex Moret\'o and Noelia Rizo. For any prime
number~$\ell$ they predict the existence of bijections with special properties
between ordinary irreducible characters of a finite group $G$ and those of
subnormalisers of $\ell$-elements of $G$. Here, for $x\in G$ its
\emph{subnormaliser in $G$} is defined as
$$\Sub_G(x):=\big\langle g\in G\mid
    \langle x\rangle\subn\langle g,x\rangle \big\rangle,$$
the subgroup generated by the elements $g\in G$ such that $\langle x\rangle$ is
subnormal in $\langle g,x\rangle$. For $\ell$-elements $x$ the subnormaliser has
a more intuitive and manageable description, namely $\Sub_G(x)$ is generated by
the normalisers of the Sylow $\ell$-subgroups of $G$ that contain~$x$
\cite[Prop.~2.6]{Ma25}.
We denote by $\Irr^x(G)$ the set of complex irreducible characters of~$G$ that
do not vanish at~$x$. The following first appeared in \cite{MR25}, see
\cite[Conj.~B]{Mo26}:

\begin{conjA}[Moret\'o]   \label{conj:AN}
 Let $G$ be a finite group and $\ell$ a prime number. Then for any
 $\ell$-element $x\in G$ there exists a bijection
 $f_x:\Irr^x(G)\to\Irr^x(\Sub_G(x))$ such that
 \begin{enumerate}
  \item[\namedlabel{deg}{\rm(1)}] $\chi(1)_\ell=f_x(\chi)(1)_\ell$, and
  \item[\namedlabel{fields}{\rm(2)}] $\QQ(\chi(x))=\QQ(f_x(\chi)(x))$.
 \end{enumerate}
\end{conjA}

That is, the bijection should preserve $\ell$-parts of character degrees as well
as irrationality properties of the character value at $x$. In this paper, we
will say that the \emph{strong subnormaliser conjecture} holds for $G$ at $x$
if furthermore
 \begin{itemize}
  \item[\namedlabel{values}{\rm(3)}] $\chi(x)=\pm f_x(\chi)(x)\quad$ for all
   $\chi\in\Irr^x(G)$,
\end{itemize}
that is, if there exists a bijection $f_x$ that even preserves the character
values at~$x$, up to sign. Clearly, \ref{values} implies \ref{fields}. It is
known, though, that \ref{values} does not hold in general, but given that the
known counter-examples only occur for groups with non-abelian Sylow
$\ell$-subgroups, one might speculate that~\ref{values} can be achieved at
least when Sylow $\ell$-subgroups of~$G$ are abelian.

The strong subnormaliser conjecture was shown to hold for all symmetric groups
$\fS_n$ by Mart{\'i}nez \cite{MM25} for the prime $\ell=2$. He also proved it
for any prime $\ell$ for those $\ell$-elements $x\in\fS_n$ contained in a unique
Sylow $\ell$-subgroup (so-called \emph{picky} elements). In joint work with
Schaeffer Fry \cite{MS25} we showed the subnormaliser conjecture for picky
$\ell$-elements of groups of Lie type in characteristic different from~$\ell$,
and the strong form when moreover $\ell=2$. For $p$-solvable groups the picky
case was shown by Moret\'o, Navarro and Rizo \cite{MNR26}.

Navarro \cite{Na04} conjectured that there exist McKay bijections with respect
to the prime~$\ell$ preserving character fields over the $\ell$-adic numbers
$\QQ_\ell$. It seems likely that the conjectured Moret\'o bijections can
also be chosen with this property.
\smallskip

In this paper we obtain further instances of the strong form of the
Moret\'o conjecture for unipotent characters of nearly simple groups of
Lie type with abelian Sylow $\ell$-subgroups, moreover preserving blocks and
realised by Galois equivariant bijections:

\begin{thmA}   \label{thm:main}
 Let $\bG$ be a simple linear algebraic group with Steinberg map $F:\bG\to\bG$.
 Let $\ell$ be a prime such that Sylow $\ell$-subgroups of $G:=\bG^F$ are
 abelian. Then the Moret\'o Conjecture holds for the unipotent characters
 of~$G$ for all $\ell$-elements $x\in G$.  \par
 More precisely, there is a bijection
 $$f:\Uch^x(G)\to\Uch^x(\Sub_G(x))$$
 as in Conjecture~\ref{conj:AN}, depending only on $\Sub_G(x)$ but not on~$x$,
 such that
 \begin{enumerate}[\rm(1)]
  \item $f(\chi)(xy) = \pm \chi(xy)$ for every $\ell'$-element
   $y\in\cent_G(x)_{\ell'}$ and every $\chi\in\Uch^x(G)$;
  \item $f$ preserves $\ell$-blocks; and
  \item for all $\chi\in\Uch^x(G)$ the character fields
   $\QQ_\ell(\chi)=\QQ_\ell(f(\chi))$ agree.
 \end{enumerate}
\end{thmA}

That is, the bijection preserves character values up to sign even on all
elements of $G$ with $\ell$-part equal to~$x$. The set $\Uch(\Sub_G(x))$ of
`unipotent characters' of the subnormaliser will be defined at the beginning
of Section~\ref{subsec:dHC}.
\smallskip

In our situation there is of course a strong connection between the existence
of a Galois equivariant bijection as above and the Alperin--McKay--Navarro
(AMN) conjecture. Indeed, if Sylow $\ell$-subgroups of $G$ are abelian, all
characters have height zero in their block by the proven, but highly
non-trivial Brauer height zero conjecture. Then, by the (unproven) AMN
conjecture applied to both the group $G$ and the subnormaliser $\Sub_G(x)$ in
question we would obtain Galois equivariant bijections
between height zero characters in corresponding blocks of~$G$ and $\Sub_G(x)$.
But whether they respect non-vanishing at $x$, and even the value at $x$ up to
sign, is not at all clear. Note also that the AMN conjecture has at present not
been reduced to a question on nearly simple groups.
\smallskip

In fact, our result extends to all very good primes in a slightly weaker form,
involving a subgroup $\hat S(\bC)$ to be defined in
Section~\ref{subsec:setting}, which still seems strong enough for an eventual
inductive approach:

\begin{thmA}   \label{thm:mainB}
 Let $\bG$ be a simple linear algebraic group with Frobenius map $F:\bG\to\bG$
 with respect to an $\FF_q$-structure. Let $\ell\ge5$ be a prime which is good
 for $\bG$, not dividing~$q$, nor $q-\eps$ if $G=A_n(\eps q)$. Let $d$ be the
 order of $q$ modulo~$\ell$. Then for any $\ell$-element $1\ne x\in\bG^F$ there
 is a $d$-split proper Levi subgroup $\bC$ of $(\bG,F)$ and a bijection
 $$f:\Uch_d(G\mid \bC)\to\Uch_d(\hat S(\bC)\mid \bC),$$
 depending only on $\bC$ but not on~$x$, such that for any
 $\chi\in\Uch_d(G\mid \bC)$
 \begin{enumerate}[\rm(1)]
  \item $\chi(1)_\ell=f(\chi)(1)_\ell$;
  \item $f(\chi)(xy) = \pm \chi(xy)$ for every $\ell'$-element
   $y\in\cent_G(x)_{\ell'}$; and
  \item the character fields $\QQ_\ell(\chi)=\QQ_\ell(f(\chi))$ agree.
 \end{enumerate}
\end{thmA}

Here, $\Uch_d(G\mid \bC)$ denotes the set of unipotent characters of $G$
lying above $d$-cuspidal pairs of $\bC$.

Let us note that even though we only treat the case of unipotent characters
here it seems reasonable to hope that this might serve as a blueprint for the
general case: using an equivariant Jordan decomposition our methods and results
might be extendable, at the cost of additional technicalities, to all
irreducible characters, at least for groups coming from an algebraic group
with connected centre, generalising our approach in \cite{MS25} for the picky
case.
\medskip

Our proofs rely on the construction of bijections to \emph{generic subnormaliser
subgroups} as introduced in \cite{Ma26} for which we prove a generalisation of
$d$-Harish-Chandra theory for unipotent characters in Theorem~\ref{thm:d-hc}.
We derive generic bijections preserving values up to sign from the
Curtis--Schewe formula for Lusztig restriction (as in \cite{MS25}) in
Theorem~\ref{thm:bij}. The existence of a bijection to subnormalisers in the
abelian Sylow case stated in
Theorem~\ref{thm:main} is then shown in Theorem~\ref{thm:MR}, using our
earlier determination of subnormalisers of $\ell$-elements in this situation
\cite{Ma26}. The block preservation and the Galois equivariance are the content
of Proposition~\ref{prop:blocks} and Theorem~\ref{thm:galois}, respectively.
Finally, in Section~\ref{subsec:rat} we also prove Theorem~\ref{thm:mainB}.
\medskip

\noindent
{\bf Acknowledgements:} I thank Alex Moret\'o and Noelia Rizo for introducing
me to this beautiful conjecture and to Mandi Schaeffer Fry for her comments on
an earlier version.

\section{Generic subnormalisers and $d$-Harish-Chandra theory}   \label{sec:RLG}
In this section we introduce the generic setup and construct generic
versions of the stipulated bijections through a suitable generalisation of
$d$-Harish-Chandra theory.

\subsection{Unipotent characters of generic subnormalisers}   \label{subsec:setting}

We consider the following setting: $\bG$ is a linear algebraic group over an
algebraically closed field of characteristic~$p$ and $F:\bG\to\bG$ is a
Frobenius endomorphism with respect to an $\FF_q$-structure, and $G:=\bG^F$,
a finite group of Lie type. We refer to \cite{MT} and \cite{GM20} for basic
notions and results on the structure and representation theory of~$\bG$,
respectively $G$.

Let $d\ge1$ be an integer and let $\bT_d\le\bG$ be a Sylow $d$-torus. We note
$W_d:=W_d(G):=\norm_G(\bT_d)/\cent_G(\bT_d)$ for its relative Weyl group. For
a $d$-split Levi subgroup $\bC$ of $(\bG,F)$ we consider the subgroups
 $$\bS:=\bS(\bC):=\langle\bC^w\mid w\in W_d\rangle$$
and
 $$\hbS:=\hbS(\bC):=\langle\norm_G(\bT_d),\bC\rangle$$
of $\bG$. These were first defined in \cite[Prop.~4.28]{Ma26}, where we showed
that $\bS$ is connected reductive and $F$-stable, $\hbS^\circ=\bS$, 
$\norm_G(\bS)=\hbS^F$, and $\bC$ is a $d$-split Levi subgroup of~$\bS$, being
the centraliser $\bC=\cent_\bS(\bZ(\bC)^\circ_d)$ of the Sylow $d$-torus
$\bZ(\bC)^\circ_d$ of its centre. Observe that the definition
of~$\bS(\bC)$ depends on the Sylow $d$-torus $\bT_d$ of $\bC$, but since all
of these Sylow tori are conjugate inside $\bC$, $\bS(\bC)$ is uniquely defined
from $\bC$ up to conjugation.

Writing $S:=\bS^F$ and $\hat S:=\hbS^F$ for the finite groups of fixed points,
we thus have the following containments of groups
$$C\le S\unlhd\hat S$$
in which the normalisers of $\bT_d$ are given by
$$\cent_G(\bT_d)\le \norm_C(\bT_d)\le \norm_S(\bT_d)\unlhd
    \norm_{\hat S}(\bT_d)=\norm_G(\bT_d)$$
all of which then also normalise $T_d$ and have quotients
\begin{equation}   \label{eq1}
  1\le W_d(C)\le W_d(S)\unlhd W_d(G)
\end{equation}
modulo $\cent_G(\bT_d)$. (Observe that $\norm_{\hat S}(\bT_d)=\norm_G(\bT_d)$
and thus $\norm_S(\bT_d)$ is indeed normal in $\norm_G(\bT_d)$.) Now $W_d(S)$,
being the relative Weyl group of the Sylow $d$-torus $\bT_d$ of $\bS$, is a
finite complex reflection group \cite[\S3A]{BMM93}, normal in the reflection
group $W_d(G)$, hence generated by some conjugacy classes of reflections in
$W_d(G)$. The possibilities for such a constellation were classified in
\cite[Prop.~3.12]{BMM99}.

\begin{rem}   \label{rem:gen subn}
The above construction is strongly linked to subnormalisers: assume that
$x\in\bT_d^F$ is an $\ell$-element for a prime $\ell$ such that $q$ has
order~$d$ modulo~$\ell$, Sylow $\ell$-subgroups of $\bG^F$ are abelian and
$\bC=\cent_\bG(x)^F$. Then $\hat S$ contains $\Sub_G(x)$ by
\cite[Cor.~4.29]{Ma26}. In fact, in most cases the two groups agree; one should
think of $\hat S$ as a kind of \emph{generic subnormaliser} of~$x$. This will
come into play in our proof of Theorem~\ref{thm:main} in Section~\ref{sec:conj}.
\end{rem}

\begin{rem}   \label{rem:all C}
 In \cite{Ma26} we determined the structure of $\bS(\bC)$ and $\hbS(\bC)$ in
the case that $\bC$ is the centraliser of an $\ell$-element of $\bG^F$ for a
prime $\ell$ with $d=e_\ell(q)$ and such that $\bG^F$ has abelian Sylow
$\ell$-subgroups. This result does in fact cover all $d$-split Levi subgroups
of~$\bG$. Indeed, let $\bC$ be $d$-split in $(\bG,F)$. Then for any $a\ge1$
prime to~$d$, $\bC$ is also $d$-split in $(\bG,F^a)$. Choosing $a$ large enough
we can arrange that $q^{ad}-1$ has a Zsigmondy prime divisor $\ell$ larger than
the order of the Weyl group of $\bG$, so that Sylow $\ell$-subgroups
of~$\bG^{F^a}$ are abelian. Enlarging $a$ even further if necessary,
the Sylow $d$-torus of $(\bZ(\bC),F^a)$ will even contain an $\ell$-element~$x$
such that $\cent_\bG(x)=\bC$. Thus, $\bC$ arises in (and hence is covered by)
the setting of \cite{Ma26}. Moreover, choosing $a$ to be prime to the order of
the graph automorphism induced by $F$ (as we may), the complete root data of
$(\bC,F)$ and $(\bC,F^a)$ are isomorphic, so give rise to the same series of
finite reductive groups. 
\end{rem}

The following property of $\hat S$ will be important for our intended
generalisation of $d$-Harish-Chandra theory and thus for the construction of a
Moret\'o bijection:

\begin{prop}   \label{prop:ext}
 Assume $\bG$ is simple. Then for all $d$-split Levi subgroups $\bC$
 of~$(\bG,F)$ all unipotent characters of~$S:=\bS(\bC)^F$ extend to their
 inertia group in $\hat S:=\hbS(\bC)^F$.
\end{prop}

\begin{proof}
We give the proof by going through the possible types of simple groups $\bG$
and their various classes of $d$-split Levi subgroups $\bC$. We start by
observing that if $\bC$ is a maximal torus, then $\bS=\bC$ and
$\hat S=\norm_G(\bT_d)$, so the unique unipotent character of~$S$ is the trivial
character for which the claim clearly holds. So we may assume that $\bC$ is not
just a torus. Secondly, if $\hat\bS$ is connected and hence equal to~$\bS$,
there is nothing to prove either.
\par
Now first consider $\bG$ such that $G=\bG^F$ is of exceptional type. As
discussed in Remark~\ref{rem:all C}, the occurring groups $\hat S(\bC)$ were
determined in \cite[Thm~3.2]{Ma26} and are
listed in Tables~1 and~2 of loc.~cit. In the cases in Table~1, $\hat S$ is
either all of~$G$, or $\bC$ is a torus, a case we already discussed. Now
consider the cases in Table~2. The group $\hat S$ (and thus also
$S=(\hat\bS^\circ)^F$) is given in Column~4. In all but the 2nd to
last line, $[\bS,\bS]$ is simple and $\hat S$ acts by a group of graph or
graph-field automorphisms on $\bS$. In this case, all unipotent characters
of~$S$ extend to their inertia group in $\hat S$ by \cite[Thm~2.4]{Ma08}.
In the penultimate line, $\bS$ is of type $A_1^3.D_4$ and $\hat S$ acts by a
group of graph automorphisms $\fS_3$ diagonally on both the $A_1^3$ and
the~$D_4$-factor. Again, it is clear that any unipotent characters extends to its
inertia group.

We now turn to $G$ of classical type. Again, the groups $\hat S$ were identified
in \cite{Ma26}. For $\bG^F=\SL_n(q)$, by \cite[Thm~4.6]{Ma26} the only relevant
case has $\hat S=(\GL_d(q)\wr\fS_a)\cap G$ with $n=ad$ and $a,d\ge1$. In the
wreath product $\GL_d(q)\wr\fS_a$ all (unipotent) characters of $\GL_d(q)^a$
extend to their inertia groups \cite[Thm~4.4.3]{JK}. Since the unipotent
characters restrict irreducibly to $\SL_d(q)^a$ this yields the required
extensions for $\hat S$.
Similarly, for $\bG^F=\SU_n(q)$, by \cite[Thm~4.10]{Ma26} the only relevant
cases are $\hat S=(\GU_e(q)\wr\fS_a)\cap G$ with $n=ae$ and $a,e\ge1$ where
$e=2d$ if $d$ is odd, $e=d/2$ if $d\equiv2\pmod4$ and $e=d$ otherwise.
We can argue as for $\SL_n(q)$.

The groups $\hat S$ in the other classical types are described in
\cite[Thms~4.17, 4.18, 4.25 and 4.26]{Ma26}. In all cases, either $\bC$ is a
maximal torus, or $\hat\bS$ is connected, or $[\bS,\bS]$ is simple, or
the group is a wreath product with a symmetric group, so we may conclude by the
above arguments.
\end{proof}

\subsection{An extension of $d$-Harish-Chandra theory}   \label{subsec:dHC}
Let $(\bL,\la)$ be a unipotent $d$-cuspidal pair of $(\bG,F)$ and write
$\cE(G,(\bL,\la))$ for the $d$-Harish-Chandra series of $G$ above $(\bL,\la)$.
Recall that by $d$-Harish-Chandra theory, for every $d$-split Levi subgroup
$\bM$ of $\bG$ containing $\bL$ there is an isometry
$$I^M=I_{\bL,\la}^M:\ZZ\Irr(W_M(\bL,\la))\longrightarrow\ZZ\cE(M,(\bL,\la)),
  \qquad\vhi\mapsto\eps_\vhi\gamma_\vhi,$$
with suitable signs $\eps_\vhi\in\{\pm1\}$ such that Lusztig induction
$R_\bM^\bG$ satisfies
\begin{equation}   \label{eq:HC}
  R_\bM^\bG\circ I^M = I^G\circ\Ind_{W_M}^{W_G},
\end{equation}
where $M:=\bM^F$ and $W_M:=W_M(\bL,\la):=\norm_G(\bL,\la)/\bL^F$ (see
\cite[Thm~3.2]{BMM93}).

Now, as above, let $\bC$ be a $d$-split Levi subgroup of $\bG$ and $(\bL,\la)$
a unipotent $d$-cuspidal pair of $\bC$. Let $W_G:=W_G(\bL,\la)$
be the  relative Weyl group of $(\bL,\la)$ in $\bG$. Then the relative Weyl
group $W_C(\bL,\la)$ of $(\bL,\la)$ in $\bC$ is a parabolic subgroup of the
complex reflection group $W_G$ since $\bC$ is $d$-split in $\bG$.
As before, we get the following containments of groups
$$1\le W_C:=W_C(\bL,\la)\le W_S:=W_S(\bL,\la)\unlhd W_G=W_G(\bL,\la).$$
(The case $\bL=\cent_\bG(\bT_d)$, $\la=1$ already appeared in~(\ref{eq1})
above.)
Let $\Uch(G):=\cE(G,1)$ denote the set of unipotent characters of $G$,
and similarly $\Uch(S):=\cE(S,1)$ for $\bS=\bS(\bC)$. We define $\Uch(\hat S)$
as the set of irreducible characters of $\hat S$ lying above unipotent
characters of~$S$. Let us denote by 
$\cE(\hat S,(\bL,\la))\subseteq\Uch(\hat S)$ the set of irreducible characters
of $\hat S$ lying above the unipotent $d$-Harish-Chandra series
$\cE(S,(\bL,\la))$ of $S$ above $(\bL,\la)$. We claim that there is a
`disconnected' $d$-Harish-Chandra theory for $\hat S$ mirroring the one
inside~$\bG$ as recalled above.

To prove this, we need some preparations. Observe that for $n\ge2$ the
imprimitive complex reflection group $G(n,1,2)$ is naturally a normal reflection
subgroup of $G(2n,2,2)$ of index~2 (see the table in \cite[Prop.~3.12]{BMM99},
for example). With respect to this embedding we have:

\begin{lem}   \label{lem:ext}
 All irreducible characters of $W_1:=G(n,1,2)$ extend to $W:=G(2n,2,2)$.
\end{lem}

\begin{proof}
The imprimitive 2-dimensional complex reflection group $W$ has a normal abelian
subgroup of index~2, so all of its irreducible characters possess degree at
most~2 (and thus the same holds for its subgroup $W_1$). Hence, all irreducible
characters of $W_1$ of degree~2 do extend to~$W$.
The number of linear characters of $W_1$ and of $W$ equals the order of the
respective commutator factor group, and this is equal to $2n$, $4n$
respectively. Thus, the linear characters of $W_1$ do extend as well.
\end{proof}

\begin{thm}   \label{thm:d-hc}
 Let $\bG$ be a simple algebraic group with Frobenius map~$F$, let $\bC$ be a
 $d$-split Levi subgroup of $(\bG,F)$ for some $d\ge1$ and set
 $\hbS:=\hbS(\bC)$. Then for any unipotent $d$-cuspidal pair $(\bL,\la)$ of
 $(\bC,F)$ there is an isometry
 $$I^{\hat S}:=I^{\hat S}_{\bL,\la}:\ZZ\Irr(W_G(\bL,\la))\longrightarrow
   \ZZ\cE(\hat S,(\bL,\la)),\qquad \vhi\mapsto\hat\eps_\vhi\hat\ga_\vhi,$$
 with signs $\hat\eps_\vhi\in\{\pm1\}$, such that
 $$R_\bC^\bG 
   = I^G\circ(I^{\hat S})^{-1}\circ\Ind_S^{\hat S}\circ R_\bC^\bS\qquad
   \text{on $\ \cE(C,(\bL,\la))$}.$$
\end{thm}

Note that in the special case $\bS=\cent_\bG(\bT_d)$ (and thus $\bC=\bS$ and
$\hat S=\norm_G(\bT_d)$) this is just the usual assertion of $d$-Harish-Chandra
theory.
 
\begin{proof}
As $\bC$ is $d$-split in $\bG$, we have
$R_\bC^\bG\circ I^C = I^G\circ\Ind_{W_C}^{W_G}$ by $d$-Harish-Chandra theory
for $\bG$, see~(\ref{eq:HC}). Thus, composing with the signed bijections
$(I^G)^{-1}$ and $I^C$, it suffices to show the existence of a signed bijection
$I^{\hat S}$ with
$$I^{\hat S}\circ\Ind_{W_C}^{W_G} = \Ind_S^{\hat S}\circ R_\bC^\bS\circ I^C.$$
Now, since $\bC$ is $d$-split in $\bS$ as well, (\ref{eq:HC}) for
$\bS$ gives $R_\bC^\bS\circ I^C = I^S\circ\Ind_{W_C}^{W_S}$, so our claim
reduces to showing the existence of a bijection $I^{\hat S}$ with signs with
$$I^{\hat S}\circ\Ind_{W_S}^{W_G} = \Ind_S^{\hat S}\circ I^S.$$
That is, induction from $S$ to $\hat S$ decomposes as induction from $W_S$
to $W_G$, up to signs. Now $W_S$ is normal in $W_G$ and $W_G/W_S\cong \hat S/S$.
Thus the claim follows if for all $\vhi\in\Irr(W_S)$, $\vhi$~and
$I^S(\vhi)\in\cE(S,(\bL,\la))$ have the same inertia group and do extend to
it. The extension property for $S\le\hat S$ is
shown in Proposition~\ref{prop:ext}. Furthermore, if
$[\bS,\bS]$ is simple and $\hbS$ acts by graph automorphisms, generalised
$d$-Harish-Chandra theory for $\hat\bS$ was proved in \cite[Thm~4.6]{Ma07}.

By the classification of the groups $\hat\bS$ in \cite{Ma26} for $G$ of
exceptional type it remains to discuss the case $\bG=E_7$ with $d=4$. Here we
have $C=\Phi_4A_1(q)^3A_1(q^2)$, whose unipotent 4-cuspidal pairs are of the
form $(\bL,\la)$ with $\bL=\cent_\bG(\bT_4)$ and $\la$ running over the
unipotent characters of $A_1(q)^3$. The relative Weyl groups of these
$(\bL,\la)$ are then $W_S=G(4,2,2)$, and $W_G=G_8$ if $\la$ is $W_d$-invariant,
and $W_G=G(4,1,2)$ if not (see \cite[Tab.~1]{BMM93}). It can then be verified
by inspection that the signed bijections $I^S:\Irr(W_S)\to\cE(S,(\bL,\la))$
are such that corresponding inertia groups in $W_G$, respectively in~$\hat S$
are isomorphic. Furthermore, the characters of $W_S$ extend to their inertia
groups as $W_G/W_S$ has cyclic Sylow subgroups.

For $G$ of classical type, the remaining cases are the wreath products, as well
as the groups $\GL_n(q).2$, $\GU_n(q).2$ and $\GO_n^+(q^2).2\cap G$ in
$G=\SO_{2n}^+(q)$. First assume $G=\SL_n(q)$ with $S=\GL_d(q)^a\cap G$ and
$\hat S=(\GL_d(q)\wr\fS_a)\cap G$ where $n=ad$. Here, we have
$C=(\GL_1(q^d)^{a-1}\times\GL_d(q))\cap G$ (see \cite[Thm~4.6]{Ma26}). As noted
above, the
assertion is clear for the minimal $d$-cuspidal pair $(\cent_\bG(\bT_d),1)$.
All other unipotent $d$-cuspidal pairs in $\bC$ are of the form $(\bL,\la)$ with
$\bL=\bC$ and $\la=(1^{\boxtimes a-1}\boxtimes\la_a)$ where $\la_a$ is labelled
by a $d$-core partition $\mu_a$ of~$d$ (see e.g. \cite[Exmp.~3.5.29]{GM20}).
For any such pair,
$$W_S=W_S(\bL,\la)=C_d^{a-1},\quad
  W_G=W_G(\bL,\la)=G(d,1,a-1)=C_d\wr\fS_{a-1}.$$
Let $\chi\in\cE(\GL_d(q)^a,(\bL,\la))$, then $\chi$ is labelled by an $a$-tuple
of partitions $(\mu_1,\ldots,\mu_a)$ where $\mu_1,\ldots,\mu_{a-1}\vdash d$
have trivial $d$-core and $\mu_a$ is as above. That is,
$\mu_1,\ldots,\mu_{a-1}$ are $d$-hooks.
Let $\vhi\in\Irr(C_d^{a-1})$ be parametrised by $(\mu_1,\ldots,\mu_{a-1})$.
Then the stabiliser of $\chi$ in $\hat S/S=\fS_a$ is the same Young subgroup of
$\fS_{a-1}$ as the stabiliser of $\vhi$ in $W_G=C_d\wr\fS_{a-1}$, and by the
character theory of wreath products, $\vhi$ extends to this stabiliser, as
wanted. The situation for the wreath products in the other classical types is
entirely analogous.

For $\hat S=\GL_n(q).2$, $\GU_n(q).2$ the claim follows by \cite[Thm~4.6]{Ma07}.

Finally, assume $G=\SO_{2n}^+(q)$ from \cite[Thm~4.25(8)]{Ma26}. Thus,
$d=n\equiv0\pmod4$, $C=\GU_2(q^e)$ with $e:=d/2$, $S=\SO_n^+(q^2)$ and $\hat S$
is an extension of $\GO_n^+(q^2)$ of degree~2. The unique unipotent
$d$-cuspidal pair of $\bC$ is the minimal one, $(\cent_\bG(\bT_d),1)$. So
$W_S=G(n,2,2)$ in $W_G=G(2n,2,2)$ (see \cite[Exmp.~3.5.29]{GM20}). Let
$W_1:=G(n,1,2)\unlhd G(2n,2,2)$ (see \cite[Prop.~3.12]{BMM99}). Then $W_1$
acts on $\Irr(W_S)$ as $\GO_n^+(q^2)$ acts on $\cE(\SO_n^+(q^2),(\bT_d,1))$
by \cite[Thm~4.6]{Ma07}. By Lemma~\ref{lem:ext} all characters of $W_1$ extend
to $W_G$. Thus, the inertia groups of characters of $W_S$ in $W_G$ agree with
those of the corresponding characters in $\cE(\SO_n^+(q),(\bT_d,1))$.
\end{proof}

\section{The subnormaliser conjecture for unipotent characters}   \label{sec:conj}

In this section, we prove Theorems~\ref{thm:main} and~\ref{thm:mainB}.
Let $\bG$ be a simple algebraic group with a Frobenius endomorphism
$F:\bG\to\bG$.

\begin{rem}   \label{rem:defchar}
 The assertion of Theorem~\ref{thm:main} holds trivially when $\ell$ is the
 defining characteristic of $\bG$. Indeed, since Sylow $\ell$-subgroups of $G$
 are assumed to be abelian, $\bG$ is of type $A_1$, and all non-trivial
 $\ell$-elements $x$ of $G$ are regular unipotent and hence picky (see
 \cite[Thm~3.2]{Ma25}). Thus $\Sub_G(x)=B$ is a Borel subgroup of $G$. The only
 unipotent character of $G$ not vanishing on non-trivial $\ell$-elements of $G$
 is the trivial character, and so all of our claims hold trivially by letting
 $\Uch(B):=\{1_B\}$ and $f:1_G\mapsto 1_B$.
\end{rem}

For the remainder of the section we may and will hence assume that $\ell$ is a
non-defining prime for~$\bG$.

\subsection{Value preserving bijections}

\begin{thm}   \label{thm:bij}
 Let $\bG$ be simple and $F:\bG\to\bG$ a Frobenius map with respect to an
 $\FF_q$-structure. Let $\ell{\not|}q$ be a prime, $d:=e_\ell(q)$ and
 $\bT_d\le\bG$ a Sylow $d$-torus. Then for every $d$-split Levi subgroup~$\bC$
 of $(\bG,F)$ and unipotent $d$-cuspidal pair $(\bL,\la)$ of $\bC$ there is a
 bijection
 $$f:\cE(G,(\bL,\la))\to\cE(\hat S(\bC),(\bL,\la))$$
 such that for every $\ell$-element $x\in\bT_d^F$ with
 $\cent_\bG^\circ(x)\le\bC$
 $$f(\chi)(xy) = \pm \chi(xy)\qquad
   \text{for every $\chi\in\cE(G,(\bL,\la))$ and $y\in \cent_G(x)_{\ell'}$}.$$
 This bijection preserves $\ell$-parts of character degrees.
\end{thm}

\begin{proof}
For simplicity write $\bS:=\bS(\bC)$ and $\hbS:=\hbS(\bC)$.
Let $(\bL,\la)$ be a unipotent $d$-cuspidal pair of $\bC$ (and hence also of
$\bS$ and $\bG$). We claim that $f:=\pm I^{\hat S}\circ(I^G)^{-1}$, a bijection
according to Theorem~\ref{thm:d-hc} has the desired properties. Here, the sign
is chosen such that $f(\chi)(1)>0$ for $\chi\in\cE(G,(\bL,\la))$. So let
$\chi\in\cE(G,(\bL,\la))$ and let $\hat\chi:=f(\chi)\in\cE(\hat S,(\bL,\la))$.
Then for any $g=xy$, with $x\in(\bT_d)_\ell^F$ an $\ell$-element with
$\cent_\bG^\circ(x)\le\bC$ and $y\in\bC_{\ell'}^F$, we have
$$\hat\chi(g)=\hat\chi|_S(g)
  =\sum_{\ga\in\cE(S,(\bL,\la))}\langle \hat\chi|_S,\ga\rangle\,\ga(g)$$
by the definition of $\cE(\hat S,(\bL,\la))$.
Now for any $\ga\in\cE(S,(\bL,\la))$, since
$\cent_\bS^\circ(g)\le \cent_\bG^\circ(x)\le\bC$ we have
$$\ga(g)=\tw*R_\bC^\bS(\ga)(g)
  =\sum_{\mu\in\cE(C,(\bL,\la))}\blangle\tw*R_\bC^\bS(\ga),\mu\brangle\,\mu(g)
  =\sum_{\mu\in\cE(C,(\bL,\la))}\blangle\ga,R_\bC^\bS(\mu)\brangle\,\mu(g),$$
by the Curtis--Schewe formula \cite[Cor.~3.3.13]{GM20}, where the second
equality follows as all constituents of $\tw*R_\bC^\bS(\ga)$ lie in
$\cE(C,(\bL,\la))$ by $d$-Harish-Chandra theory, and the third by adjointness.
Inserting this into the previous formula, we find
\begin{equation}   \label{eq:CS2}
\begin{aligned}
 \hat\chi(g)=&\sum_{\ga\in\cE(S,(\bL,\la))}\langle \hat\chi|_S,\ga\rangle
  \sum_{\mu\in\cE(C,(\bL,\la))}\blangle\ga,R_\bC^\bS(\mu)\brangle\,\mu(g)\\
  =&\sum_{\mu\in\cE(C,(\bL,\la))}\Big\langle\hat\chi|_S,\sum_{\ga\in\cE(S,(\bL,\la))}
  \blangle\ga,R_\bC^\bS(\mu)\brangle\,\ga\Big\rangle\,\mu(g)\\
  =&\sum_{\mu\in\cE(C,(\bL,\la))}\blangle\hat\chi|_S,R_\bC^\bS(\mu)\brangle\,\mu(g)
  =\sum_{\mu\in\cE(C,(\bL,\la))}\blangle\hat\chi,\Ind_S^{\hat S}(R_\bC^\bS(\mu))\brangle\,\mu(g),
\end{aligned}\end{equation}
again using that all constituents of $R_\bC^\bS(\mu)$ lie in $\cE(S,(\bL,\la))$
and Frobenius reciprocity. On the other hand, again by the Curtis--Schewe
formula applied in $G$, for $\chi$ we get the value
\begin{equation}   \label{eq:CS}
  \chi(g)=\tw*R_\bC^\bG(\chi)(g)
  =\sum_{\mu\in\cE(C,(\bL,\la))}\blangle\tw*R_\bC^\bG(\chi),\mu\brangle\,\mu(g)
  =\sum_{\mu\in\cE(C,(\bL,\la))}\blangle\chi,R_\bC^\bG(\mu)\brangle\,\mu(g).
\end{equation}
As $\hat\chi=f(\chi)=\pm(I^{\hat S}\circ (I^G)^{-1})(\chi)$, comparison of
(\ref{eq:CS}) with~(\ref{eq:CS2}) shows $\chi(g)=\pm\hat\chi(g)$ by
Theorem~\ref{thm:d-hc}, which establishes the existence of a bijection~$f$
as claimed.

It remains to compare the $\ell$-parts of the character degrees. Let
$\vhi\in\Irr(W_G(\bL,\la))$ and $\chi:=\pm I^G(\vhi)\in\cE(G,(\bL,\la))$,
$\hat\chi:=\pm I^{\hat S}(\vhi)\in\cE(\hat S,(\bL,\la))$ so that
$\hat\chi=f(\chi)$. By the degree formula for $d$-Harish-Chandra series
\cite[Cor.~4.6.25]{GM20} the degree polynomials of $\la$ and of $\chi$ are
divisible by the same power of $\Phi_d$, and the same holds for the degree
polynomial of any character $\chi'\in\cE(S,(\bL,\la))$. Now by
\cite[Prop.~6.5]{KMS} (the proof shows that it holds for any $d$-Harish-Chandra
series) we have $(\chi(1)/\la(1))_\ell=\vhi(1)_\ell$, and similarly
$(\chi'(1)/\la(1))_\ell=\vhi'(1)_\ell$, where $\chi'=\pm I^S(\vhi')$. 
By Theorem~\ref{thm:d-hc} we also have $\hat\chi(1)/\chi'(1)=\vhi(1)/\vhi'(1)$
if $\vhi'$ lies under $\vhi$ (and so $\chi'$ under $\chi$), hence
$$\chi(1)_\ell=(\vhi(1)\la(1))_\ell=(\vhi'(1)\la(1)\hat\chi(1)/\chi'(1))_\ell
  =\hat\chi(1)_\ell,$$
showing preservation of $\ell$-parts of degrees.
\end{proof}

In \cite{MS25} we considered the special case where
$\bS=\bC=\cent_\bG(\bT_d)$ and $\hbS=\norm_\bG(\bT_d)$. Note that in contrast
to that picky case, here it can happen that $\cE(G,(\bL,\la))$ contains
characters of degree divisible by $\ell$; 
this is the case when their Harish-Chandra vertex $(\bL,\la)$ is such that
$\bL$ is not a minimal $d$-split Levi subgroup or if its relative Weyl group
has such characters of degree divisible by~$\ell$. In particular the bijection
constructed above is more than just a McKay bijection.

We can now show the existence of bijections preserving character values. As for
$\Irr(G)$, for an element $x\in G$ we denote by $\Uch^x(G)$ those unipotent
characters that do not vanish at $x$. 
The following is the first part of Theorem~\ref{thm:main}:

\begin{thm}   \label{thm:MR}
 Let $\bG$ be a simple linear algebraic group with Steinberg map $F:\bG\to\bG$.
 Let $\ell$ be a prime such that Sylow $\ell$-subgroups of $G:=\bG^F$ are
 abelian. Then there is a bijection
 $$f:\Uch^x(G)\to\Uch^x(\Sub_G(x))$$
 as in Conjecture~\ref{conj:AN}, depending only on $\Sub_G(x)$ but not on~$x$,
 such that
 $$f(\chi)(xy) = \pm \chi(xy)\qquad
   \text{for every $\ell'$-element $y\in\cent_G(x)_{\ell'}$}.$$
 In particular, the strong form of the Moret\'o conjecture holds at $\ell$
 for unipotent characters of $\bG^F$.
\end{thm}

\begin{proof}
Let $\bG$ and $G=\bG^F$ be as in the statement. Let $\ell$ be a prime such
that Sylow $\ell$-subgroups of $G$ are abelian. By Remark~\ref{rem:defchar}
we may assume $\ell$ is different from the defining characteristic of $\bG$.
Let $x\in G$ be a non-trivial $\ell$-element. If $F$ is not a Frobenius map,
then $G$ is a Suzuki or Ree group. In this case, by \cite[Thm~5.11]{Ma25}
either $\Sub_G(x)=G$ (and there is nothing to prove), or $x$ is picky,
$\cent_G(x)=\bT_\Phi^F$ and $\Sub_G(x)=\norm_G(\bT_\Phi)$ for some Sylow
$\Phi$-torus $\bT_\Phi\le\bG$, where $\Phi$ is a cyclotomic polynomial
determined by $\ell$. Here, usual $\Phi$-Harish-Chandra theory provides a
signed bijection $\Irr(W_G(\bT_\Phi))\to\cE(G,(\bT_\Phi,1))$, to which the
arguments in the proof of Theorem~\ref{thm:bij} apply verbatim. Here, all
characters are of $\ell'$-degree.

So now assume $F$ is a Frobenius map with respect to an $\FF_q$-structure. Let
$d:=e_\ell(q)$ and $\bT_d\le\bG$ a Sylow $d$-torus. By \cite[Prop.~2.4]{Ma14}
the assumption that Sylow $\ell$-subgroups of $G$ are abelian implies that
$\bT_d$ contains a
Sylow $\ell$-subgroup of $G:=\bG^F$, and furthermore that the centraliser
$\bC:=\cent_\bG(x)$ is a $d$-split Levi subgroup of~$\bG$. By the results in
\cite{Ma26}, the only situation where $\Sub_G(x)\ne\hat S:=\hat S(\bC)$ occurs
in~$G=G_2(4)$, with $\ell=5$, where $\Sub_G(x)$ is the sporadic Janko group
$J_2$. In this case, the assertion of Theorem~\ref{thm:main} can be checked
from the known character tables.

Thus, finally, we may assume $\Sub_G(x)=\hat S$, that is, the subnormaliser
coincides with the generic subnormaliser. Let $\ga\in\Uch^x(G)$ and $(\bL,\la)$
be a unipotent $d$-cuspidal pair such that $\ga\in\cE(G,(\bL,\la))$. By the
Curtis--Schewe formula~(\ref{eq:CS}), $\ga(x)\ne0$ implies that $(\bL,\la)$ is
contained in $\bC$, up to conjugation, and similarly any
$\hat\ga\in\Uch^x(\Sub_G(x))$ must lie in $\cE(\hat S,(\bL,\la))$ for such an
$(\bL,\la)$, by (\ref{eq:CS2}). Thus, we are in the situation of
Theorem~\ref{thm:bij}. By inspection of the structures from \cite{Ma26} the
$\hat S$-classes and the $G$-classes
of unipotent $d$-cuspidal pairs of $\bC$ are in bijection in all cases, and for
each such, Theorem~\ref{thm:bij} provides a bijection preserving $\ell$-parts
of character degrees and values up to sign, which implies~\ref{values} and
thus~\ref{fields} of Conjecture~\ref{conj:AN}.
\end{proof}

\subsection{Blocks}

We next prove the preservation of $\ell$-blocks.

\begin{prop}   \label{prop:blocks}
 The bijection constructed in Theorem~\ref{thm:MR} respects $\ell$-blocks.
\end{prop}

\begin{proof}
By Remark~\ref{rem:defchar} we may assume $\ell$ is a non-defining prime.
First assume $F$ is a Frobenius map. For any $d$-split Levi subgroup $\bM\le\bC$
we write
$R_\bM^{\hbS}:=\Ind_S^{\hat S}\circ R_\bM^\bS$. Then Theorem~\ref{thm:d-hc}
says that $R_\bM^{\hbS}$ satisfies a $d$-Harish-Chandra theory in the sense of
\cite[Def.~2.9]{KM13} above any unipotent $d$-cuspidal pair $(\bL,\la)$
of~$\bC$. Now observe that $R_\bM^\hbS$ commutes with generalised decomposition
maps; this is clear for $\Ind_S^{\hat S}$ and holds for $R_\bM^\bS$ by
\cite[Prop.~3.3.17]{GM20}. Furthermore, by our assumption on~$\ell$ all
centralisers of $\ell$-elements in $\bS$ are $d$-split Levi subgroups, and thus
equal to the centraliser of the $\ell$-part of their centre. It then follows
that the arguments in the proof of \cite[Prop.~2.17]{KM13} apply to our present
situation. Thus, all constituents of $R_\bL^\hbS(\la)$ lie in the same
$\ell$-block of $\hat S$. That is, all elements of $\cE(\hat S,(\bL,\la))$ lie
in the same $\ell$-block of $\hat S$. Also, by \cite[Thm~5.24]{BMM93}, all
elements of $\cE(G,(\bL,\la))$ lie in the same $\ell$-block of $G$.

On the other hand, by the same reference, unipotent characters of $G$ above
non-conjugate $d$-cuspidal pairs of $G$ lie in different blocks. Assume that
$\ga_i\in\cE(\hat S,(\bL_i,\la_i))$ lie in the same $\ell$-block, for 
unipotent $d$-cuspidal pairs $(\bL_i,\la_i)$ of $\bC$, $i=1,2$.
Then there are some characters $\chi_i\in\Irr(S|\ga_i)$ lying in the same
$\ell$-block of $S$. But all characters of $S$ below $\ga_i$ lie in
$\cE(S,(\bL_i,\la_i))$ by definition. Then \cite[Thm~5.24]{BMM93} applied
to~$S$ shows that
$(\bL_1,\la_1)$ and $(\bL_2,\la_2)$ must be conjugate in $S$ and hence in
$\hat S$, so $\ga_1,\ga_2$ lie in the same $d$-Harish-Chandra series. The claim
is then clear from our construction of $f$ in Theorem~\ref{thm:MR} using the
maps between $d$-Harish-Chandra series in Theorem~\ref{thm:bij}.

If $F$ is not a Frobenius map, so $G$ is a Suzuki or Ree group we can argue in
precisely the same way, since again there exists a $d$-Harish-Chandra theory
for unipotent characters by \cite{BMM93}.
\end{proof}

\subsection{Rationality properties}   \label{subsec:rat}
We complete the proof of Theorem~\ref{thm:main} by discussing the Galois
equivariance over $\QQ_\ell$ of our bijections.

\begin{thm}   \label{thm:galois}
 In the situation of Theorem~\ref{thm:main}, the bijection
 $f:\Uch^x(G)\to\Uch^x(\Sub_G(x))$ can be chosen such that in addition
 $\QQ_\ell(\chi)=\QQ_\ell(f(\chi))$ for all $\chi\in\Uch^x(G)$.
\end{thm}

\begin{proof}
Let $\bG$ be simple with Steinberg map $F:\bG\to\bG$.  By
Remark~\ref{rem:defchar}, $\ell$ is not the underlying characteristic
of~$\bG$. If $F$ is not a Frobenius map, then as discussed in the proof of
Theorem~\ref{thm:MR}, the map $f$ just comes from ordinary $d$-Harish-Chandra
theory. Here the claim follows from \cite[Folg.~6.7]{BM93} (or can be checked
easily).

Thus, we may assume $F$ is a Frobenius map with respect to an $\FF_q$-structure.
Let $\bT_d\le\bG$ be a Sylow $d$-torus where $d=e_\ell(q)$ and let
$x\in T_d:=\bT_d^F$ be an $\ell$-element. Let $\bC:=\cent_\bG(x)$ and
$\hat S=\Sub_G(x)$. By the description of $\Uch^x(G)$ in the proof of
Theorem~\ref{thm:MR} the set $\Uch^x(G)$ consists only of characters lying in
$d$-Harish-Chandra series above unipotent $d$-cuspidal pairs $(\bL,\la)$
of~$\bC$. We discuss the occurring situations case by case.

First assume $G:=\bG^F$ is of exceptional type. In the non-generic case with
$\Sub_G(x)=J_2$ in $G_2(4)$ at $\ell=5$ it is easily verified from the known
character tables (see \cite{GAP}) that all unipotent characters involved are
rational. (In fact, in this case the whole bijection from
\cite[Prop.~3.3]{Ma26} can be chosen Galois equivariant over $\QQ$: on either
side, eight characters are rational, and the other ten take values in
$\QQ(\sqrt{5})$.)

In the generic case $\Sub_G(x)=\hat S$, the subnormalisers can be found in
\cite[Tab.~1 and~2]{Ma26}. If $\bC$ is a maximal torus and so
$\hat S=\norm_G(\bT_d)$ then $(\bL,\la)=(\bC,1)$, and thus $\Uch(\hat S)$ is
in bijection with
$\Irr(W_G(\bT_d))$ by deflation. Now the irrationalities of $\cE(G,(\bC,1))$ and
$\Irr(W_G(\bT_d)$ are well-known (see \cite[Prop.~4.5.5]{GM20}) and comparing
these the claim can be seen to hold (this was in fact already checked in
\cite[Satz~6.6, Folg.~6.7]{BM93}). Thus, it remains to discuss the cases in
\cite[Tab.~2]{Ma26}. Four essentially different situations occur here.
Either $\hat S$ is an extension of a group of type $A_2$ by a graph or
graph-field automorphism of order~2, or $\hat S$ is an extension of a group of
type~$D_4$ by a group of graph automorphisms of order~3 or 6, or $\hat S$ is an
extension of a group of type $E_6$ by a graph or graph-field automorphism of
order~2, or $\hat S$ is an extension of a group of type~$D_4$ by the Weyl group
of $G_2$. In all but the last case, the irrationalities of $\Uch(\hat S)$ are
described in \cite[Thm~3]{DM24}, while again the irrationalities of
$\Uch(G)$ are well-known \cite[Prop.~4.5.5]{GM20}. From this, the claim is
seen to hold.

In the fourth case, $G=E_8(q)$, $d=8$ and $\hat S=D_4(q^2).W(G_2)$. Here,
$\hat S/S\cong W(G_2)$ acts on $S=D_4(q^2)$ by a group of graph automorphisms
$\fS_3$ times
a group $C_2$ generated by a field automorphism $\ga$. All unipotent characters
of $S$ extend to $\QQ_\ell$-rational characters of $S\langle\ga\rangle$ by
\cite[Thm~7.6]{RSST}, as all Frobenius eigenvalues are equal to~1. There exist
8 unipotent characters $\rho$ of $S$ invariant under $\fS_3$; they have one
rational extension to $S.\fS_3$ by \cite[Thm~3]{DM24} and thus all three
characters above such $\rho$ are rational-valued. Above each of the two orbits
of unipotent characters of length~3, there is, of course, just one character
of~$S.\fS_3$, necessarily rational.
Thus, all unipotent characters of $\hat S$ are rational valued. By
\cite[Prop.~4.5.5]{GM20}, all but two of the unipotent characters of $E_8(q)$
in the principal 8-series are rational, while the other two have values in
$\QQ(\sqrt{-1})$. Since $\ell|(q^4+1)$ here, all characters are thus
$\QQ_\ell$-rational, as required.
\par
We now investigate the groups $G$ of classical type. Here, all unipotent
characters of $G$ are rational \cite[Prop.~4.5.5]{GM20}, so it suffices to
argue that all elements of $\cE(\hat S,(\bL,\la))$ are $\QQ_\ell$-rational.
It is shown in \cite[Thm~8.8]{RSST} that for any $d$-cuspidal pair $(\bL,\la)$
of $\bG$ there is a bijection
$\Irr(\norm_{\hat G}(\bL)|\la)\to\cE(\hat G,(\bL,\la))$, equivariant under the
absolute Galois group of $\QQ_\ell$. Thus our claim holds for unipotent
characters in $d$-Harish-Chandra series $(\bL,\la)$ with
$\hat S=\norm_{\hat G}(\bL)$. This covers, in particular, the case when
$\bS=\cent_\bG(\bT_d)$ is minimal $d$-split.

We consider the remaining configurations. First assume $G=\SL_n(q)$.
In the case when $\Sub_G(x)=(\GL_{ad}(q)\times\GL_r(q))\cap G$ with $n=ad+r$,
all unipotent characters of $\Sub_G(x)$ are rational, being restrictions of
unipotent characters of $\GL_{ad}(q)\times\GL_r(q)$. By \cite[Thm~4.6]{Ma26}
this leaves us with the case $\Sub_G(x)=(\GL_d(q)\wr\fS_a)\cap G$, $n=ad$.
Here, by the character theory of wreath products (see \cite[Thm~4.4.8]{JK}),
all characters of $\GL_d(q)\wr\fS_a$ above the (rational) unipotent
characters of $\GL_d(q)$ are rational, and thus so are their (irreducible)
restrictions to $\Sub_G(x)$. The same arguments apply to the subnormalisers
in $G=\SU_n(q)$ described in \cite[Thm~4.10]{Ma26}. For $G=\Sp_{2n}(q)$ when
$q$ is even, and for $\SO_{2n+1}(q)$, \cite[Thms~4.17 and~4.18]{Ma26} show
that there occur further subnormalisers of the form
$\GO_{ae}^\pm(q)\times\Sp_{2r}(q)$, or $\GO_{ae}^\pm(q)\times\GO_{2r+1}(q)$.
By \cite[Thm~3]{DM24} all invariant unipotent characters of $\SO_{ae}^\pm(q)$
have a rational extension to $\GO_{ae}^\pm(q)$, while the group $\GO_{2r+1}(q)$
is a direct product of $\SO_{2r+1}(q)$ with $C_2$, so all unipotent characters
of these subnormalisers are again rational. This also deals with the
subnormalisers in $G=\SO_{2n}^-(q)$ described in \cite[Thm~4.27]{Ma26}.

Finally, assume $G=\SO_{2n}^+(q)$. The above arguments allow us to handle
all types of subnormalisers listed in \cite[Thm~4.25]{Ma26}, except for
$\Sub_G(x)=\GL_n(q).2$, $\GU_n(q).2$ or $\GO_n^+(q^2).2$. More precisely,
the case
$\GL_n(q).2$ occurs only if the order $d$ of $q$ modulo~$\ell$ is odd. In this
case, $q$ has a square root modulo~$\ell$, and hence also in $\QQ_\ell$. The
case $\GU_n(q).2$ occurs only if $d\equiv2\pmod4$, so the order of $-q$
modulo~$\ell$ is odd, and thus $-q$ has a square root in $\QQ_\ell$. But under
these conditions all unipotent characters of $\Sub_G(x)$ are $\QQ_\ell$-rational
by \cite[Thm~3]{DM24}. 
The subnormaliser $\Sub_G(x)=\GO_n^+(q^2).2$ is an extension of $\SO_n^+(q^2)$
by the direct product of a group of graph automorphisms of order~2 with a group
of field automorphisms of order~2. Here, we may conclude by applying
\cite[Prop.~8.6]{RSST} twice, once for $G$ and once for $\Sub_G(x)$.
\end{proof}

This completes the proof of Theorem~\ref{thm:main}.

\begin{proof}[Proof of Theorem~\ref{thm:mainB}]
Assume the situation of Theorem~\ref{thm:mainB}, so $\ell\ge5$ is good for
$\bG$ and $\ell$ does not divide $q-\eps$ if $G=A_n(\eps q)$,
$\eps\in\{\pm1\}$. Then
$\ell\in\Gamma(\bG,F)$ in the notation of \cite{CE94}, and so by
\cite[Prop.~2.5]{CE94}, for any $\ell$-elements $1\ne x\in\bG^F$ there exists
a proper $d$-split Levi subgroup $\bC$ of $(\bG,F)$ such that
$\cent_\bG^\circ(x)\le \bC$.  The existence
of a bijection satisfying the preservation of character values up to sign now
follows by Theorem~\ref{thm:bij}, again using that $\chi\in\Uch^x(G)$ implies
that $\chi$ lies above a unipotent $d$-cuspidal pair of~$\bC$. For the
assertion about rationality, by Remark~\ref{rem:all C} any $d$-split Levi
subgroup $\bC$ is among those discussed in Theorem~\ref{thm:galois} and thus
the Galois equivariance follows from that result.
\end{proof}


\end{document}